\documentclass[12pt]{article}

\usepackage{amssymb}
\usepackage{color,amsmath,amssymb, amsfonts,amstext,amsthm}

\usepackage{epsfig, graphicx, graphics}
\usepackage[colorlinks=false, linktocpage=true]{hyperref}
\usepackage{longtable}

\newcommand{\p}{\partial}

\renewcommand{\phi}{\varphi}
\newcommand{\e}{\epsilon}

\newcommand{\R}{{\mathbb R}}

\newcommand{\EX}{{\mathbb{E}}}
\newcommand{\PX}{{\mathbb{P}}}
\newcommand{\cF}{\mathcal{F}}

\newtheorem{theorem}{Theorem}
\newtheorem{lemma}{Lemma}
\newtheorem{remark}{Remark}

\title{Ensemble Averaging for Dynamical Systems under Fast Oscillating Random Boundary Conditions}

\author{Wei Wang$^1$, Jian Ren$^2$, Jinqiao Duan$^3$ and Guowei He$^4$     \\
1. Department of Mathematics, Nanjing University\\Nanjing, China   \\
 \emph{E-mail:  wangweinju@yahoo.com.cn   }\\
2. School of Mathematics and Statistics\\Huazhong University of Science and Technology\\Wuhan 430074, China  \\
 \emph{E-mail:  renjian0371@gmail.com    }\\
3. Institute for Pure and Applied Mathematics, University of California\\
Los Angeles, CA 90095, USA\\
\emph{E-mail: jduan@ipam.ucla.edu }\\ \& \\
   Department of Applied Mathematics, Illinois Institute of Technology \\
  Chicago, IL 60616, USA \\
 \emph{E-mail:  duan@iit.edu }\\
4. Laboratory for Nonlinear Mechanics\\Institute of Mechanics, Chinese Academy of Sciences\\
 Beijing 100080, China\\  \emph{E-mail:     hgw@lnm.imech.ac.cn }
 }

\begin{document}
\date\today

\maketitle

\pagestyle{plain}

\begin{abstract}

This paper is devoted to provide a theoretical underpinning for ensemble forecasting with rapid fluctuations in body forcing and in boundary conditions. Ensemble averaging principles are proved under suitable `mixing' conditions on random boundary conditions and on random body forcing. The ensemble averaged model is a nonlinear stochastic partial differential equation, with the deviation process (i.e., the approximation error process)  quantified as the solution of a linear stochastic partial differential equation.

\medskip

 {\bf Short Title:} Ensemble Averaging under Random Boundary Conditions    \\



\end{abstract}

\section{Motivation}  \label{motive}

A complex system often involves with multiple scales, uncertain parameters or coefficients,  and fluctuating  interactions with its environment.
Ensemble forecasting for such a complex system is a   prediction method   to obtain   collective or ensembled view of its dynamical evolution,  by generating    multiple numerical predictions   using   different   but plausible realizations of a model for the system.   The multiple simulations are generated to account for errors introduced by  sensitive dependence on the initial/boundary conditions and  errors introduced due to imperfections in the model \cite{Palmer1, Palmer2}.

In order to better understand the theoretical foundation of ensemble forecasting, we consider a   system modeled by partial differential equations~(PDEs) with fast oscillating random forcing in the physical medium or on the physical boundary, and show that ensemble averaged dynamics converges to the original dynamics, as a  scale parameter tends to zero.

Some relevant recent works~\cite{CSW, DIP, Pardoux2012, Watanabe} are about averaging or   homogenization   for  random partial differential equations (random PDEs) with fast oscillating coefficients in time or   space.  Different from the method in the above mentioned works, we present a more direct approach, in order to derive an averaging principle and deviation estimates for PDEs with random oscillating coefficients   and  random oscillating  boundary conditions.
We previously studied ~\cite{WangDuan09} stochastic partial differential equations~(stochastic PDEs) with perturbed white noise dynamical boundary conditions which are measured by a small scale parameter $\e>0$. In that case, the effectively reduced model does not capture the influence  of  the random force on boundary.

In the present paper,  we  consider the following PDE with a random oscillating   body forcing and/or a small random oscillating boundary condition, for a unknown random field $u^\e(x, t, \omega)$
 \begin{eqnarray}\label{rpde888}
&&u^\e_t=u^\e_{xx} + g(t/\e, u^\e, \omega)  \,, \quad u^\e(x, 0)=u_0(x), \\
&&u^\e(0, t)=\sqrt{\e} f(t/\e, \omega), \; \quad  u^\e(l, t)=0.\label{rpde888-bc}
\end{eqnarray}
Here  $x \in (0, l)$, $l>0$, $t>0$, $\e$ is a small positive scale parameter, and $\omega$ is in a sample space $\Omega$.   A probability $\PX$ with a $\sigma-$algebra $\cF$ is defined on this sample space. The mathematical expectation with respect to $\PX$ is denoted by $\EX$. We often suppress the $\omega-$dependence for notational clarity. 



We prove an ensemble averaging theorem~(Theorem \ref{thm:AVE-DBC} in \S \ref{boundary888}) for the random system   (\ref{rpde888})--(\ref{rpde888-bc}), i.e., we obtain a stochastic model for $u^\e$ as $\e \to 0$.  It turns out that the random    boundary condition appears as a white noise on the dynamical field equation for $u$, as $\e\rightarrow 0$\,.   The ensemble averaged model is a stochastic partial differential equation (stochastic PDE) for $u$, instead of a random PDE, with a homogenous  boundary condition. When the random boundary condition is absent, we further show that the deviation process (i.e., approximation error process), $u^\e -u$, can be quantified as the solution of a linear stochastic PDE (Theorem \ref{thm:AVE-DEV-RPDE}  in \S \ref{body888}).

On the technical side, in order to pass the limit $\e\rightarrow 0$\,, we  first prove the tightness of the distribution of $\{u^{\e}\}$\,, so we   just consider $\langle u^{\e}, \phi\rangle$ for every bounded continuous function $\phi$ with compact support. Then, in \S \ref{sec:ens-AVE}, we construct a process $\mathcal{M}^{\e}_{t}$,   which is a martingale by Ethier and Kurtz's result~\cite[Proposition 2.7.6]{EK86}.   This construction is  very direct~\cite{DIP}. By passing the limit $\e\rightarrow 0$ in $\mathcal{M}_{t}^{\e}$, we obtain the stochastic PDE satisfied by the limit $u$ of $u^{\e}$\,.
This method is  also applied to show that the deviation process, $u^\e -u$, is the solution of a linear stochastic PDE; see \S ~\ref{body888}.  


Note that we take $\sqrt{\e}$ as the intensity scale for the noise boundary condition. This is  for simplicity.  In fact our approach can also treat the case  $\e^{\frac{\alpha}{2}\wedge 1}f(t/\e^{\alpha})$\,, with $0<\alpha \leq 2$\,. A   similar case is also discussed in \cite{DIP}. But the case $\alpha>2$
 is more  singular, one should consider the limit of  $\e^{\frac{\alpha}{2}-1}u^{\e}$ as $\e\rightarrow 0$\,.


This paper is organized as follows. After recalling some basic background, we prove an ensemble averaging theorem for a random PDE system with a random boundary condition and  with a random body forcing, in \S \ref{boundary888},  and further characterize the deviation process in \S \ref{body888}.

\section{Ensemble averaging under small fast oscillating  random  boundary conditions}
\label{boundary888}

 We consider the random PDE system (\ref{rpde888})--(\ref{rpde888-bc}).  Consider the Hilbert space $H=L^{2}(0, l)$ with  the usual norm   $\|\cdot\|_{0}$ and   inner product   $\langle\cdot, \cdot \rangle$. Define $A=\p_{xx}$ with the zero Dirichlet boundary condition. It   defines a compact analytic semigroup $S(t)$\,, $t\geq 0$\,, on $H$. Denote by $0<\lambda_{1}\leq \lambda_{2}\leq \cdots$ the eigenvalues of $-A$  with the corresponding eigenfunctions $\{e_{k}\}_{k=1}^{\infty}$\,, which forms an   orthonormal   basis of $H$\,. For every $\alpha>0$,  define a new norm $\|u\|_{\alpha}=\|(-A)^{\alpha/2}u\|_{0}$, for those $u\in H$ such that this quantity is finite.


 Here we make the following assumptions about the mixing properties of the random boundary and body forcing in the random PDE system (\ref{rpde888})--(\ref{rpde888-bc}).
\begin{description}
  \item[(\textbf{H}$_{g}$)] For every $t$\,, $g(t, \cdot)$ is Lipschitz continuous in $u$ with Lipschitz constant~$L_{g}$ and $g(t, 0)=0$\,. For every $u\in H$\,, $g(\cdot, u)$ is an  $H$-valued  stationary random process and is strongly mixing with an exponential rate $\gamma>0$. That is,
  \begin{equation*}
    \sup_{s\geq 0}\sup_{U\in\mathcal{G}_0^s\,, V\in\mathcal{G}_{s+t}^\infty}|\mathbb{P}(U\cap V)-\mathbb{P}(U)\mathbb{P}(V)|    \leq e^{-\gamma t}\,, \quad t\geq 0,
  \end{equation*}
  where $0\leq s\leq t\leq \infty$\,,  and $\mathcal{G}_s^t=\sigma\{g(\tau, u): s\leq \tau\leq t\}$\, is the $\sigma$-algebra generated by $\{g(\tau, u): s\leq \tau\leq t\}$\,.
  \item[(\textbf{H}$_{f}$)]
  The process $f(t)$ is a bounded continuous differentiable process with $|f(t)|\leq C_{f}$\,, for some constant $C_{f}>0$\,,  and the time derivative process $f_{t}(t)$ is a bounded stationary process with $\mathbb{E}f_{t}=0$ and the mixing rate is  exponential\,. That is,
  \begin{equation*}
    \sup_{s\geq 0}\sup_{U\in\mathcal{F}_0^s\,, V\in\mathcal{F}_{s+t}^\infty}|\mathbb{P}(U\cap V)-\mathbb{P}(U)\mathbb{P}(V)| \leq  e^{-\lambda t}\,, \quad t\geq 0,
  \end{equation*}
  where $0\leq s\leq t\leq\infty$\,, $\lambda>0$, and  $\mathcal{F}_s^t=\sigma\{f_{t}(\tau): s\leq \tau\leq t\}$\, is the $\sigma$-algebra generated by $\{f_{t}(\tau): s\leq \tau\leq t\}$\,.
\end{description}

\begin{remark}
A simple example of such $f_{t}$ is the stationary solution of the following linear stochastic equation
\begin{equation*}
d\eta=-\lambda\eta\,dt+dB(t),
\end{equation*}
where $B(t)$ is a standard scalar Brownian motion.
\end{remark}

\begin{remark}\label{rem:DBC}
Taking  time derivative on the random boundary condition, we have
\begin{eqnarray*}
&&u^\e_t=u^\e_{xx} + g(t/\e, u^\e)  \,, \quad u^\e(x, 0)=u_0\\
&&u^\e_{t}(0, t) = \tfrac{1}{\sqrt{\e}} f_{t}(t/\e), \;  \quad u^\e_{t}(l, t)=0,
\end{eqnarray*}
which is a   system with a random dynamical boundary condition.
\end{remark}

To `homogenize' the inhomogeneous  boundary condition in the system (\ref{rpde888})--(\ref{rpde888-bc}),  we transform the random boundary condition to the field  equation by introducing a new random field $\hat u^\e=u^\e-\sqrt{\e}f(t/\e)(1-\frac{x}{l})$. Then, $\hat
u^\e_{xx}=u^\e_{xx}$ and the system (\ref{rpde888})--(\ref{rpde888-bc}) becomes
\begin{eqnarray}
&&\hat u^\e_t=\hat u^\e_{xx}+g(t/\e, u^{\e})-\tfrac{1}{\sqrt{\e}}f_t(t/\e)(1-\tfrac{x}{l}), \label{new888} \\
&&{}\hat u^\e(x, 0)=u_0-\sqrt{\e}f(0)(1-\tfrac{x}{l}),\\ &&{}\hat u^\e(0,t)=0,\,\, \hat u^\e(l,t)=0,
\end{eqnarray}
which is a random system with homogeneous boundary conditions.  By the assumption (\textbf{H}$_{f}$), $f$  is bounded. Thus  for every $t\geq 0$ and $x\in(0, l)$\,,
\begin{equation}\label{e:hatu-u}
\hat{u}^{\e}-u^{\e}=\sqrt{\e}f(t/\e)(1-\tfrac{x}{l})\rightarrow 0\,, \quad \e\rightarrow 0\,.
\end{equation}
So in the following subsections, we consider $\hat{u}^{\e}$\,, and derive an ensemble averaged equation to be satisfied by the limit of $\hat{u}^{\e}$\,.

\subsection{Tightness}\label{sec:tight}

In this section, we examine the tightness of the distribution of $\hat{u}^{\e}$ in space of continuous functions, $C(t_0, T; H)$, for all fixed $T>t_0>0$\,.

 In the mild or integral formulation, the equation (\ref{new888}) becomes
\begin{equation*}
\hat{u}^{\e}(t)=S(t)\hat{u}^{\e}(0)+\int_{0}^{t}S(t-s)g(\tfrac{s}{\e}, u^{\e}(s))\,ds-\tfrac{1}{\sqrt{\e}}\int_{0}^{t}S(t-s)f_{t}(\tfrac{s}{\e})(1-\tfrac{x}{l})\,ds\,.
\end{equation*}
By the properties of the semigroup $S(t)$, we have
\begin{equation*}
\|\hat{u}^{\e}(t)\|_{0}\leq \|\hat{u}^{\e}(0)\|_{0}+\int_{0}^{t}\|g(\tfrac{s}{\e}, u^{\e}(s))\|_{0}\,ds+\tfrac{1}{\sqrt{\e}}\left\|\int_{0}^{t}S(t-s)f_{t}(\tfrac{s}{\e})(1-\tfrac{x}{l}) ds  \right\|_{0}\,,
\end{equation*}
and by the assumption $(\textbf{H}_{g})$\,,
\begin{equation}\label{e:g-L0}
\|g(\tfrac{s}{\e}, u^{\e}(s))\|_{0}\leq L_{g}\|u^{\e}(s)\|_{0}\leq L_{g}(\|\hat{u}^{\e}(s)\|_{0}+\sqrt{\e l}C_{f})\,.
\end{equation}
Then we have,  for every $T>0$ and $0<t\leq T$\,,
\begin{eqnarray}
\mathbb{E}\sup_{0\leq s\leq t}\|\hat{u}^{\e}(s)\|_{0}&\leq& \|\hat{u}^{\e}(0)\|_{0}+L_{g}\int_{0}^{t}\sup_{0\leq r\leq s}\|\hat{u}^{\e}(r)\|_{0}\,ds\nonumber\\&&{}+C_{T,1}+\sup_{0\leq s\leq t}\|I^{\e}(s)\|_{0}, \label{e:E-u}
\end{eqnarray}
where
$C_{T,1}$ is a positive constant   depending only on $L_{g}$\,, $l$ and $C_{f}$\,, and
\begin{equation*}
I^{\e}(t):=\tfrac{1}{\sqrt{\e}}\int_{0}^{t}S(t-s)f_{t}(\tfrac{s}{\e})(1-\tfrac{x}{l})\,ds\,.
\end{equation*}
Next we treat the singular term $I^{\e}(t)$\,.
By the factorization method~\cite{PZ92}, for some $0<\alpha<1$\,,
\begin{equation*}
I^{\e}(t)=\tfrac{\sin\pi\alpha}{\alpha}\int_{0}^{t}(t-s)^{\alpha-1}S(t-s)Y^{\e}
(s)\,ds,
\end{equation*}
with
\begin{equation}\label{e:Y-eps}
Y^{\e}(s)=\tfrac{1}{\sqrt{\e}}\int_{0}^{s}(s-r)^{-\alpha}f_{t}(\tfrac{r}{\e})S(s-r)(1-\tfrac{x}{l})\,dr\,.
\end{equation}
Then, for every $T>0$, there is a positive constant $C_{T,2}$ such that
\begin{equation*}
\mathbb{E}\sup_{0\leq s\leq t}\|I^{\e}(s)\|^{2}_{0}\leq C_{T,2}\int_{0}^{t}\mathbb{E}\|Y^{\e}(s)\|^{2}_{0}\,ds\,,\quad 0\leq t\leq T\,.
\end{equation*}
Notice that
\begin{eqnarray*}
\mathbb{E}\|Y^{\e}(s)\|_{0}^{2}&=&\tfrac{1}{\e}
\int_{0}^{s}\int_{0}^{s}(s-r)^{-\alpha}(s-\tau)^{-\alpha}\mathbb{E}\left[f_{t}(\tfrac{r}{\e})f_{t}(\tfrac{\tau}{\e})\right]\\
&&{}\times S(s-r)(1-\tfrac{x}{l})S(s-\tau)(1-\tfrac{x}{l})\,drd\tau\,,
\end{eqnarray*}
by the assumption $(\textbf{H}_{f})$\,.  For every $T>0$\,, there is a positive constant $C_{T,3}$ such that for all $0\leq t\leq T$,
\begin{equation}\label{e:I-eps-0}
\mathbb{E}\sup_{0\leq s\leq t}\|I^{\e}(s)\|_{0}\leq C_{T,3}\,.
\end{equation}
Hence, for every $T>0$, applying the Gronwall inequality to (\ref{e:E-u})\,, we obtain
\begin{equation}
\mathbb{E}\sup_{0\leq t\leq T}\|\hat{u}^{\e}(t)\|_{0}\leq C_{T}(1+\|\hat{u}^{\e}(0)\|_{0}),
\end{equation}
for some constant $C_{T}>0$\,. Furthermore,  from  the mild form of $\hat{u}^{\e}$\,, by the fact that $\|S(t)u\|_{1}\leq \frac{1}{\sqrt{t}}\|u\|_{0}$\,, we have
\begin{equation}\label{e:u-H1}
\|\hat{u}^{\e}(t)\|_{1}\leq \tfrac{1}{\sqrt{t}}\|\hat{u}^\e(0)\|_{0}+\int_{0}^{t}\tfrac{1}{\sqrt{t-s}}\|g(\tfrac{s}{\e}, u^{\e}(s))\|_{0}\,ds+\|I^{\e}(s)\|_{1}\,.
\end{equation}
We now consider the term $\|I^{\e}(s)\|_{1}$\,.  Still by the factorization method,
\begin{eqnarray*}
\|I^{\e}(t)\|_{1}&\leq & \tfrac{\sin\alpha}{\pi}\int_{0}^{t}(t-s)^{\alpha-1}\|S(t-s)Y^{\e}(s)\|_{1}\,ds \\
&\leq &\tfrac{\sin\alpha}{\pi}\int_{0}^{t}(t-s)^{\alpha-1}\tfrac{1}{\sqrt{t-s}}\| Y^{\e}(s)\|_{0}\,ds,
\end{eqnarray*}
where $Y^{\e}(s)$ is defined by (\ref{e:Y-eps})\,.
Then, choose $\alpha$ with $1/2<\alpha<1$, and by the same discussion for (\ref{e:I-eps-0}), we conclude that for every $T>0$
\begin{equation*}
\mathbb{E}\|I^{\e}(t)\|_{1}\leq C_{T,5}\,, \quad 0\leq t\leq T
\end{equation*}
for some constant $C_{T,5}>0$\,. Then for $t_{0}>0$, from
(\ref{e:g-L0}) and (\ref{e:u-H1}), and by Gronwall inequality  we have
\begin{equation}
\mathbb{E}\|\hat{u}^{\e}(t)\|_{1}\leq  C_{T}\,,\quad t_{0}\leq t\leq T,
\end{equation}
for some constant $C_{T}>0$\,.

To show the tightness of the distributions of $\hat{u}^{\e}$, we need a   H\"older estimate in time.  For every $0\leq s<t\leq T$\,,
\begin{eqnarray*}
&&\|\hat{u}^{\e}(t)-\hat{u}^{\e}(s)\|_{0}\\&\leq&
\|(S(t)-S(s))\hat{u}^{\e}(0)\|_{0}+\left\|\int_{s}^{t}S(t-\sigma)g(\tfrac{\sigma}{\e}, u^{\e}(\sigma))\,d\sigma\right\|_{0}\\&&{}
+\tfrac{1}{\sqrt{\e}}\left\|\int_{s}^{t}S(t-\sigma)f_{t}(\tfrac{\sigma}{\e})(1-\tfrac{x}{l})\,d\sigma \right\|_{0}\\&&{}+\left\|\int_{0}^{s}[S(t-\sigma) -S(s-\sigma)]g(\tfrac{\sigma}{\e}, u^{\e}(\sigma))\,d\sigma\right\|_{0}\\&&{}
\tfrac{1}{\sqrt{\e}}\left\|\int_{0}^{s}[S(t-\sigma) -S(s-\sigma)]f_{t}(\tfrac{\sigma}{\e})(1-\tfrac{x}{l})\,d\sigma \right\|_{0}\,.
\end{eqnarray*}
By the estimate on $\|\hat{u}^{\e}(t)\|_{0}$ and   (\ref{e:g-L0})\,, we have for some constant $C_{T}>0$\,,
\begin{eqnarray*}
\mathbb{E}\left\|\int_{s}^{t}S(t-\sigma)g(\tfrac{\sigma}{\e}, u^{\e}(\sigma))\,d\sigma\right\|_{0}\leq C_{T}\sqrt{t-s}.
\end{eqnarray*}
Moreover, by the  strong continuity of the semigroup $S(t)$, we also  have
\begin{equation*}
\mathbb{E}\left\|\int_{0}^{s}[S(t-\sigma) -S(s-\sigma)]g(\tfrac{\sigma}{\e}, u^{\e}(\sigma))\,d\sigma\right\|_{0}\leq C_{T}\sqrt{t-s}\,.
\end{equation*}
Now consider for the singular terms. First notice that $(1-x/l)$ is smooth. We have
\begin{equation*}
S(t-\sigma)(1-\tfrac{x}{l})\in L^{\infty}(0, l)\,.
\end{equation*}
Therefore,
\begin{eqnarray*}
&&\mathbb{E}\tfrac{1}{\e}\left\|\int_{s}^{t}\int_{s}^{t}S(t-\sigma)f_{t}(\tfrac{\sigma}{\e})(1-\tfrac{x}{l})\,d\sigma \right\|^{2}_{0}\\&=&
\int_{s}^{t}\int_{s}^{t}\tfrac{1}{\e}\mathbb{E}[f_{t}(\tfrac{\sigma}{\e})f_{t}(\tfrac{\tau}{\e})]\int_{0}^{l}S(t-\sigma)(1-\tfrac{x}{l})S(t-\tau)(1-\tfrac{x}{l})\,dx\,d\sigma d\tau\\
&\leq &C_{l, T}\int_{s}^{t}\int_{s}^{t}\tfrac{1}{\e}\mathbb{E}[f_{t}(\tfrac{\sigma}{\e})f_{t}(\tfrac{\tau}{\e})]\,d\sigma d\tau,
\end{eqnarray*}
for a positive constant $C_{l, T}$   depending on $T$ and $l$\,. Now by $(\textbf{H}_{f})$, we have
\begin{equation*}
\mathbb{E}\tfrac{1}{\sqrt{\e}}\left\|\int_{s}^{t}S(t-\sigma)f_{t}(\tfrac{\sigma}{\e})(1-\tfrac{x}{l})\,d\sigma \right\|_{0}\leq C_{l, T}(t-s)\,.
\end{equation*}
Furthermore,
\begin{eqnarray*}
&&\mathbb{E}\tfrac{1}{\e}\left\|\int_{0}^{s}[S(t-\sigma) -S(s-\sigma)]f_{t}(\tfrac{\sigma}{\e})(1-\tfrac{x}{l})\,d\sigma \right\|^{2}_{0}\\
&=&\tfrac{1}{\e}\sum_{k}l_{k}\int_{0}^{s}\int_{0}^{s}\mathbb{E}f_{t}(\tfrac{\sigma}{\e})f_{t}(\tfrac{\tau}{\e}) [e^{-\lambda_{k}(t-\sigma)}-e^{-\lambda_{k}(s-\sigma)}]\\&&{}\times [e^{-\lambda_{k}(t-\tau)}-e^{-\lambda_{k}(s-\tau)}]\,d\sigma\,d\tau,
\end{eqnarray*}
where
\begin{equation*}
l_{k}=\int_{0}^{l}(1-\tfrac{x}{l})e_{k}(x)\,dx\,.
\end{equation*}
Then still by $(\textbf{H}_{f})$,  we have for some constant $C_{T}>0$
\begin{equation*}
\mathbb{E}\tfrac{1}{\sqrt{\e}}\left\|\int_{0}^{s}[S(t-\sigma) -S(s-\sigma)]f_{t}(\tfrac{\sigma}{\e})(1-\tfrac{x}{l})\,d\sigma \right\|_{0}\leq C_{T}(t-s)\,.
\end{equation*}

Now we need the following lemma~\cite{Mai03}. Suppose $\mathcal{X}_{1}$~and~$\mathcal{X}_{2}$ are two Banach spaces. Let $T>0$\,, $1\leq p\leq \infty$\,, and $\mathcal{B}$~be a compact operator from~$\mathcal{X}_{1}$ to~$\mathcal{X}_{2}$. That is, $\mathcal{B}$~maps bounded subsets of~$\mathcal{X}_{1}$ to relatively compact subsets of~$\mathcal{X}_{2}$.
\begin{lemma}[\cite{Mai03}]\label{lem:compact}
Let $\mathcal{H}$ be a bounded subset of $L^{1}(0, T; \mathcal{X}_{1})$ such that $G=\mathcal{B}\mathcal{H}$ is a subset of $L^{p}(0, T; \mathcal{X}_{2})$ bounded in $L^{r}(0, T; \mathcal{X}_{2})$ with $r>1$\,. If
\begin{equation*}
\lim_{\sigma\rightarrow 0}\|u(\cdot+\sigma)-u(\cdot)\|_{L^{p}(0, T; \mathcal{X}_{2})}=0 \quad \text{uniformly for } u\in G\,,
\end{equation*}
then $G$~is relatively compact in $L^{p}(0, T; \mathcal{X}_{2})$ (and in $C(0, T; \mathcal{X}_{2})$ if $p=+\infty$).
\end{lemma}
Then by the above lemma, with $\mathcal{X}_{1}=H_{0}^{1}(0, l)$ and $\mathcal{X}_{2}=L^{2}(0, l)$ and $\mathcal{B}$ being the embedding from $\mathcal{X}_{1}$ to $\mathcal{X}_{2}$, we have the following tightness result.

\begin{lemma} \label{tight888} (Tightness)\\
Assume that both $(\textbf{H}_{g})$ and  $(\textbf{H}_{f})$ hold. For every $0<t_{0}<T$, the distribution of $\{\hat{u}^{\e}\}_{0<\e\leq 1}$ is tight in space $C(t_{0}, T; H)$\,.
\end{lemma}

\subsection{Ensemble averaging}\label{sec:ens-AVE}

Next we use the weak convergence method~\cite{Kushner} to pass the limit $\e\rightarrow 0$\,.  In this approach we construct a martingale which has the following form
\begin{equation*}
\Phi(t)-\int_{0}^{t}A^{\e}\Phi(s)\,ds,
\end{equation*}
for some $\mathcal{F}_{0}^{t}$-process $\Phi(t)$ defined by $z_{1}^{\e}(t)$ and $A^{\e}$\,, which is a pseudo differential operator to be introduced later.

Because of the tightness of $\hat{u}^{\e}$ in space $C(t_{0}, T; H)$ for every fixed $t_{0}>0$\,,  in order to determine the limit equation of $\hat{u}^{\e}$ in space $C(t_{0}, T; H)$, we   consider the limit of $\Phi(\langle \hat{u}^{\e}(t), \varphi\rangle)$, for every bounded second order differentiable function $\Phi: \R\rightarrow\R$ and  for every compactly supported smooth function $\varphi\in C^{\infty}_{b}(0, l)$\,.

First, we have
\begin{eqnarray}
&&\Phi(\langle \hat{u}^{\e}(t), \varphi\rangle)-\Phi(\langle u_{0}, \varphi\rangle)  \nonumber  \\
&=&\int_{0}^{t}\Phi'(\langle \hat{u}^{\e}(s), \varphi\rangle)\langle \hat{u}^{\e}(s), \varphi_{xx}\rangle\,ds+\int_{0}^{t}\Phi'(\langle \hat{u}^{\e}(s), \varphi\rangle)\langle g(\tfrac{s}{\e}, \hat{u}^{\e}(s)), \varphi\rangle\,ds  \nonumber \\
&&\quad{}-\tfrac{1}{\sqrt{\e}}\int_{0}^{t}\Phi'(\langle \hat{u}^{\e}(s), \varphi\rangle)\langle f_{t}(\tfrac{s}{\e})(1-\tfrac{x}{l}), \varphi\rangle\,ds\,. \label{martingale888}
\end{eqnarray}
To treat the singular term in (\ref{martingale888}), we apply a perturbation method in~\cite[Chapter 7]{Kushner}.  To this end, we define the following two processes
\begin{equation}
F_{0}^{\e}(t) :=\tfrac{1}{\sqrt{\e}} \int_{t}^{\infty}
\mathbb{E}[f_{t}^{\e}(\tfrac{s}{\e})|\mathcal{F}_{0}^{t/\e}]ds
\end{equation}
and
\begin{eqnarray}
F^{\e}_{1}(t)&:=&\tfrac{1}{\sqrt{\e}}\mathbb{E}\left[\int_{t}^{\infty}\Phi'(\langle \hat{u}^{\e}(t), \varphi\rangle)  \langle f_{t}(\tfrac{s}{\e})(1-\tfrac{x}{l}), \varphi\rangle\,ds \Big|\mathcal{F}^{t/\e} _{0}\right]\nonumber\\
&=&\Phi'(\langle \hat{u}^{\e}(t), \varphi\rangle)  \langle 1-\tfrac{x}{l}, \varphi\rangle F_{0}^{\e}(t).
\end{eqnarray}
Then we have the following lemma.
\begin{lemma}\label{lem:est-F-eps}
Assume that {\rm $(\textbf{H}_{f})$} holds.  Then
\begin{equation*}
\mathbb{E}|F^{\e}_{1}(t)|\leq C\sqrt{\e},
\end{equation*}
for some constant $C>0$ and
\begin{equation*}
\mathbb{E}\sup_{0\leq t\leq T}|F^{\e}_{1}(t)|\rightarrow 0\,,\quad \e\rightarrow 0\,,
\end{equation*}
for every $T>0$\,.
\end{lemma}
\begin{proof}
By the boundedness of $f_{t}$ and the strong mixing property, we have 
\begin{equation}\label{e:F-0}
|F_{0}^{\e}(t)|\leq C\sqrt{\e}
\end{equation}
for some constant $C>0$\,.
 Then by the choice of $\Phi$,  the proof is complete.
\end{proof}

Now we apply a diffusion approximation to derive the limit of $\hat{u}^{\e}$ in the sense of distribution. For this we introduce the following operator
\begin{equation}\label{e:A-eps}
A^{\e}\Phi(t)=\mathbb{P}\text{-}\lim_{\delta\rightarrow 0} \tfrac{1}{\delta}\mathbb{E} \left[\Phi(t+\delta)-\Phi(t)|\mathcal{F}_{0}^{t/\e}  \right]
\end{equation}
for  $\mathcal{F}_{0}^{t/\e}$ measurable function $\Phi(t)$ with $\sup_{t}\mathbb{E}|\Phi(t)|<\infty$\,.     Using   Ethier and Kurtz's result~\cite[Proposition 2.7.6]{EK86}\,, we know that
\begin{equation*}
\Phi(t)-\int_{0}^{t}A^{\e}\Phi(s)\,ds
\end{equation*}
is a martingale with respect to $\mathcal{F}_{0}^{t/\e}$\,.  Define processes $Y^{\e}$ and $Z^{\e}$ as follows
\begin{equation*}
Y^{\e}(t)=\Phi(\langle \hat{u}^{\e}(t), \varphi\rangle)-F^{\e}_{1}(t)\,,\quad Z^{\e}(t)=A^{\e}Y^{\e}(t)\,.
\end{equation*}
A direct calculation yields
\begin{eqnarray}\label{e:Z}
&&Z^{\e}(t)\\&=&\Phi'(\langle \hat{u}^{\e}(t), \varphi\rangle)\big[\langle \hat{u}^{\e}(t), \varphi_{xx}\rangle+\langle g(t/\e, u^{\e}(t)), \varphi\rangle\big]\nonumber\\
&&{}-\Phi''(\langle \hat{u}^{\e}(t), \varphi\rangle)\langle (1-\tfrac{x}{l}), \varphi\rangle^{2}\tfrac{1}{\sqrt{\e}} f_{t}(\tfrac{t}{\e}) F_{0}^{\e}(t)\nonumber\\&&{}+\Phi''(\langle \hat{u}^{\e}(t),\varphi\rangle )\Big[\langle \hat{u}^{\e}(t), \varphi_{xx}\rangle+\langle g(\tfrac{t}{\e}, u^{\e}(t)), \varphi\rangle\Big] \langle 1-\tfrac{x}{l}, \varphi\rangle F_{0}^{\e}(t)\nonumber\\
&:=&Z_{1}^{\e}(t)+Z_{2}^{\e}(t)+Z_{3}^{\e}(t)\,.\nonumber
\end{eqnarray}

Next we pass the limit $\e\rightarrow 0$ for $\hat{u}^{\e}(t)$ in space $C(t_{0}, T; H)$\,.  By the convergence result  of Walsh~\cite[Theorem~6.15]{Walsh86},  we only need to consider finite dimensional distributions of $\{\langle \hat{u}^{\e}(t), \varphi_{1}\rangle, \ldots, \langle \hat{u}^{\e}(t), \varphi_{n}\rangle\}$ for every $\varphi_{1}\,,\ldots\,, \varphi_{n}\in C_{b}^{\infty}(0, l)$\,. That is, we pass limit $\e\rightarrow 0$ in
\begin{equation*}
\mathbb{E}\left\{\left[Y^{\e}(t)-Y^{\e}(s)-\int_{s}^{t}Z^{\e}(r)\,dr\right]h(\langle \hat{u}^{\e}(r_{1}), \phi_{1}\rangle, \ldots, \langle \hat{u}^{\e}(r_{n}), \phi_{n} \rangle)\right\}=0
\end{equation*}
for every bounded continuous function $h$ and $0<r_{1}<\cdots <r_{n}<T$ with any $T>0$\,. Denote by $\hat{u}$ one limit point in the sense of distribution of $\hat{u}^{\e}$ as $\e\rightarrow 0$ in space $C(t_{0}, T; H)$\,.  For simplicity we assume $\hat{u}^{\e}$ converges in distribution to $\hat{u}$ as $\e\rightarrow 0$\,.
Then by the estimates in Lemma \ref{lem:est-F-eps} we have
\begin{equation}
Y^{\e}(t)-Y^{\e}(s)\rightarrow  \Phi(\langle \hat{u}(t), \phi\rangle)-\Phi(\langle \hat{u}(s), \phi\rangle)
\end{equation}
in distribution.

\medskip

Consider the integral term in (\ref{martingale888}).  First we need the following lemma whose proof is given in Appendix \ref{app:A1}.

\begin{lemma}\label{lem:A1}
The following convergence in probability holds:
\begin{equation*}
 \int_{0}^{t}\left[g(r/\e, u^{\e}(r))-\bar{g}(u^{\e}(r))\right]\,dr \rightarrow 0\,,\quad \e\rightarrow 0\,.
\end{equation*}
\end{lemma}
Then by this lemma, we have
\begin{equation}
\int_{s}^{t}Z_{1}(r)\,dr\rightarrow \int_{s}^{t}\Phi(\langle \hat{u}(r), \phi\rangle)\left[\langle \hat{u}(r), \phi_{xx} \rangle+\langle \bar{g}(\hat{u}(r)), \phi\rangle\right]\,dr
\end{equation}
in distribution as $\e\rightarrow 0$\,.
By the  the estimate~(\ref{e:F-0})\,, we have
\begin{equation}
\mathbb{E}\int_{s}^{t}|Z_{3}^{\e}(r)|\,dr\rightarrow 0\,.
\end{equation}
Now we consider $Z_{2}^{\e}(t)$\,. Define a bilinear operator
\begin{equation}\label{e:Sigma}
\langle\Sigma \phi, \phi \rangle=b\int_{0}^{l}\int_{0}^{l}(1-\tfrac{x}{l})\phi(x)(1-\tfrac{y}{l})\phi(y)\,dx\,dy,
\end{equation}
where $b$ is the variance of $f_t$, which is constant defined as 
\begin{equation}
b :=\mathbb{E}f_{t}(t)f_{t}(t)>0\,. \label{bbb}
\end{equation}
We again apply a perturbation method. Set
\begin{equation}
F_{2}^{\e}(t) :=-\Phi^{''}(\langle\hat{u}^{\e}(t), \phi \rangle)\langle(1-\tfrac{x}{l}), \phi\rangle^{2}\tfrac{1}{\sqrt{\e}}\int_{t}^{\infty}\mathbb{E}\big[ f_{t}(\tfrac{s}{\e})F_{0}^{\e}(s)-\tfrac12b\big|\mathcal{F}_{0}^{t/\e}\big]\,ds\,.
\end{equation}
By the properties of conditional expectation and the definition of $F_{0}^{\e}$\,,  we have
\begin{eqnarray*}
&&\tfrac{1}{\sqrt{\e}}\int_{t}^{\infty}\mathbb{E}\big[ f_{t}(\tfrac{s}{\e})F_{0}^{\e}(s)-\tfrac12b\big|\mathcal{F}_{0}^{t/\e}\big]\,ds\\&=&
\tfrac{1}{\e}\int_{t}^{\infty}\int_{s}^{\infty}\mathbb{E}\big[f_{t}(\tfrac{s}{\e})f_{t}(\tfrac{\tau}{\e})-\tfrac12 b|\mathcal{F}_{0}^{t/\e} \big]\,d\tau ds.
\end{eqnarray*}
Then, by the strong mixing  properties of $f_{t}$ in the assumption  $(\textbf{H}_{f})$\,, we have
\begin{equation*}
\sup_{t\geq 0}\mathbb{E}F_{2}^{\e}(t)=\mathcal{O}(\e)\,.
\end{equation*}
Furthermore by the same calculation as for $Z^{\e}(t)$, we have the following lemma.
\begin{lemma} 
The following result holds:
\begin{equation}
A^{\e}F_{2}^{\e}(t)=-\Phi''(\langle\hat{u}^{\e}(t), \phi \rangle)\langle (1-\tfrac{x}{l}), \phi\rangle^{2}\tfrac{1}{\sqrt{\e}}\big[f_{t}(\tfrac{t}{\e}) F_{0}^{\e}(t) -\tfrac12b\big]+\mathcal{O}(\e).
\end{equation}
\end{lemma}

Now we have the following $\mathcal{F}_{0}^{t/\e}$-martingale
\begin{eqnarray*}
\mathcal{M}_{t}^{\e}&:=&\Phi(\langle \hat{u}^{\e}(t), \phi\rangle)-F_{1}^{\e}(t)-F_{2}^{\e}(t)-\\&&{}-\int_{0}^{t}\Phi'(\langle\hat{u}^{\e}(s), \phi\rangle)\big[  \langle\hat{u}^{\e}(s), \phi_{xx} \rangle +\langle \bar{g}(\hat{u}^{\e}(s)), \phi\rangle\big]\,ds\\&&{}
+\tfrac12\int_{0}^{t}\Phi''(\langle \hat{u}^{\e}(s), \phi\rangle)\langle\Sigma \phi, \phi\rangle\,ds+\mathcal{O}(\e)\,.
\end{eqnarray*}
By passing   the limit $\e\rightarrow 0$, the distribution of the limit $u$ of $\hat{u}^{\e}$  solves the following martingale problem
\begin{eqnarray}
\mathcal{M}(\tau)&=&\Phi(\langle u(t), \phi\rangle)-\int_{0}^{t}\Phi'(\langle u(s), \phi\rangle)\big[\langle u(s), \phi_{xx}  \rangle +\langle \bar{g}(u(s)), \phi\rangle\big]
\,ds\nonumber \\&&{}+\tfrac12\int_{0}^{t}\Phi''(\langle u(s), \phi \rangle)\langle \Sigma \phi, \phi \rangle\,ds,   \label{M}
\end{eqnarray}
which is equivalent to the fact that $u$ is the martingale solution of the following stochastic PDE:
\begin{equation}\label{e:u-spde}
du=[u_{xx}+\bar{g}(u)]\,dt-\sqrt{b}(1-\tfrac{x}{l})\,dB(t),
\end{equation}
where  $B$ is a usual scalar Brownian motion, and $b$ is the variance of $f_t$ as defined in (\ref{bbb}).

Finally, by the uniqueness of the solution to equation (\ref{e:u-spde}), we have the following main result on ensemble averaging under a random boundary condition.

\begin{theorem}\label{thm:AVE-DBC} (Ensemble averaging under a random boundary condition)\\
For every $t_{0}>0$ and $T>t_{0}$\,, the solution $u^{\e}$, of the random PDE system~(\ref{rpde888}), converges in distribution to $u$ in space $C(t_{0}, T; H)$, with $u$ solving the limit equation (\ref{e:u-spde}).
\end{theorem}

\section{Ensemble averaging under fast oscillating  random body forcing}
\label{body888}

In this section, we consider the special case when the random boundary condition is absent. 
The   approach to derive ensemble averaged model in the last section   is applicable in this case.  But our goal here is to further show that the deviation process, $u^\e -u$, can be quantified as the solution of a linear stochastic partial differential equation.

We consider the following   PDE with random oscillating body forcing on a bounded interval $(0, l)$
\begin{equation}\label{e:rpde}
u^\e_t=u^\e_{xx}+f(t/\e, u^\e)\,, \quad u^\e(x, 0)=u_0, \;\;
u^\e(0,t)=0,\,\,  u^\e(l, t)=0.
\end{equation} 
  Here we make the following assumption on the random body forcing $f$.
\begin{description}
  \item[(\textbf{H})] For every $t$\,, $f(t, \cdot)$ is  continuously differentiable and Lipschitz continuous in $u$  with Lipschitz constant $L_{f}$ and $f(t, 0)=0$\,. For every $u\in H$\,, $f(\cdot, u)$ is an  $H$-valued  stationary random process and is strongly mixing with an exponential rate $\gamma>0$, i.e.,
  \begin{equation*}
    \sup_{s\geq 0}\sup_{U\in\mathcal{F}_0^s\,, V\in\mathcal{F}_{s+t}^\infty}|\mathbb{P}(U\cap V)-\mathbb{P}(U)\mathbb{P}(V)|    \leq e^{-\gamma t}\,, \quad t\geq 0.
  \end{equation*}
  Here $s$ and $t$ satisfy the condition $0\leq s\leq t\leq \infty$\,,  and $\mathcal{F}_s^t=\sigma\{f(\tau, u): s\leq \tau\leq t\}$\, is the $\sigma$-algebra generated by $\{f(\tau, u): s\leq \tau\leq t\}$\,.

\end{description}

We introduce the notation $\phi(t)$ to quantify the mixing  as follows
  \begin{equation*}
      \phi(t) \triangleq  \sup_{s\geq 0}\sup_{U\in\mathcal{F}_0^s\,, V\in\mathcal{F}_{s+t}^\infty}|\mathbb{P}(U\cap V)-\mathbb{P}(U)\mathbb{P}(V)|\,.
  \end{equation*}
By the above assumption,  for any $\alpha>0$
\begin{equation*}
\int_0^\infty\phi^\alpha(t)\,dt<\infty\,.
\end{equation*}

%

\medskip

For the random oscillating PDE \eqref{e:rpde} we have an averaging principle as above. Introduce the following averaged equation
\begin{equation}\label{e:RPDE-averaged}
u_t=u_{xx}+\bar{f}(u)\,, \quad u(0)=u_0,
\end{equation}
where $\bar{f}(u)=\mathbb{E}f(t, u)=\lim_{T\rightarrow\infty}\frac{1}{T}\int_0^T f(s/\e,
u)\,ds$. Define the deviation process
\begin{equation}
z^\e(t) :=  \tfrac{1}{\sqrt{\e}}(u^\e(t)-u(t))\,.
\end{equation}
Then  the following averaging principle will be established.

\begin{theorem}\label{thm:AVE-DEV-RPDE} (Ensemble averaging under random body forcing)\\
Assume that $({\mathbf{H}})$\, holds. Then, given a $T>0$\,, for every $u_0\in H$,  the solution $u^\e(t, u_0)$  of (\ref{e:rpde}) converges in probability  to the solution $u$ of \eqref{e:RPDE-averaged} in $C(0, T; H)$. Moreover,  the rate of convergence is   $\sqrt{\e}$\,, that is, for any $\kappa>0$ there is $C_{T}^{\kappa}>0$ such that
\begin{equation}\label{e:conver-rate}
\mathbb{P}\left \{\sup_{0\leq t\leq T}\|u^\e(t)-u(t)\|_{0}\geq  C^{\kappa}_T\sqrt{\e}\right\}\leq \kappa\,.
\end{equation}
Furthermore, the deviation process $z^\e$ converges in distribution in the space~$C(0, T; H)$ to $z$, which solves the following linear stochastic PDE
\begin{equation}\label{e:zz}
dz(t)=[z_{xx}(t)+\overline{f'}(u(t))z(t)]\,dt+d\widetilde{W}\,, \quad z(0)=z(l)=0,
\end{equation}
  where 
\begin{equation*}
\overline{f'}(u)=\mathbb{E}f'_{u}(t, u)
\end{equation*}
and  $\widetilde{W}(t)$ is an $H$-valued Wiener process defined on a new probability space $(\bar{\Omega}, \bar{\mathcal{F}}, \bar{\mathbb{P}})$ with the covariance operator
\begin{equation*}
\tilde{B}(u)=2\int_0^\infty\mathbb{E}\left[ (f(t, u)-\bar{f}(u))\otimes
(f(0, u)-\bar{f}(u)) \right]\,dt\,.
\end{equation*}
 \end{theorem}
 \begin{remark}
This deviation result is  similar  to the averaging results  for random PDEs in~\cite{DIP,ParPia}.
 
 \end{remark}

\begin{proof}
First by the assumption of  Lipschitz property on $f$ in  $(\mathbf{H})$\,, and noticing  that there is no singular term here, standard energy estimates yield that for every $T>0$
\begin{equation}\label{e:1}
\sup_{0\leq t\leq T}\|u^{\e}(t)\|_{1}^{2}\leq C_{T},
\end{equation}
and
\begin{equation}\label{e:2}
\|u^{\e}(t)-u^{\e}(s)\|_{0}\leq C_{T}|t-s|\,, \quad 0\leq s\leq t\leq T,
\end{equation}
with some positive constant $C_{T}$\,.  Then we have the tightness of the distributions of $u^{\e}$ in space $C(0, T; H)$ for every $T>0$\,.

Notice also that $u^{\e}(t)$ satisfies
\begin{equation*}
u^{\e}(t)=S(t)u_{0}+\int_{0}^{t}S(t-s)f(s/\e, u^{\e}(s))\,ds\,,
\end{equation*}
and for every $\phi\in H$
\begin{equation*}
\langle u^{\e}(t), \phi\rangle=\langle S(t)u_{0}, \phi\rangle+\int_{0}^{t}\langle S(t-s)f(s/\e, u^{\e}(s)), \phi\rangle\,ds\,.
\end{equation*}
By passing the limit $\e\rightarrow 0$\,, we  can just consider the integral term in the above equation.  By the tightness of the distributions of $u^{\e}$\,,  we can follow the same discussion as in Appendix \ref{app:A1}  which yields the
 averaged equation (\ref{e:RPDE-averaged}) and the esitmate (\ref{e:conver-rate}).\\

We next consider the deviation process $z^\e$.  By the definition of $z^{\e}$,
\begin{equation*}
\dot{z}^{\e}=z^{\e}_{xx}+\tfrac{1}{\sqrt{\e}}[f( t/\e,
u^{\e})-\bar{f}(u)]\,,\quad z^{\e}(0)=0,
\end{equation*}
with the zero Dirichlet boundary condition.
For every $\alpha>0$,
\begin{eqnarray*}
\|A^{\alpha}z^{\e}(t)\|_{0} &=&\left\|\tfrac{1}{\sqrt{\e}}\int_{0}^{t} A^{\alpha} e^{A(t-s)}[f(\tfrac{s}{\e}, u^{\e}(s))-\bar{f}(u(s))]\,ds\right\|_{0}\\
&\leq&\left\|\tfrac{1}{\sqrt{\e}}\int_{0}^{t}A^{\alpha} e^{A(t-s)}[f(\tfrac{s}{\e}, u^{\e}(s))-f(\tfrac{s}{\e},u(s))]\,ds\right\|_{0}\\
&&+\left\|\tfrac{1}{\sqrt{\e}}\int_{0}^{t} A^{\alpha}e^{A(t-s)}[f(\tfrac{s}{\e}, u(s))-\bar{f}(u(s))]\,ds\right\|_{0} \\
&:=&I_{1}(t)+I_{2}(t)\,.
\end{eqnarray*}
Notice that for $0<\alpha<1/2$\,,
\begin{equation*}
\tfrac{1}{\sqrt{\e}}\|A^{\alpha} e^{A(t-s)}[f(\tfrac{s}{\e}, u^{\e}(s))-f(\tfrac{s}{\e},u(s))]\|_{0}\leq C(1+\tfrac{1}{\sqrt{s}})L_{f}\|z^{\e}\|_{0},
\end{equation*}
for some constant $C>0$\,. Then
\begin{equation*}
\mathbb{E}\sup_{0\leq t\leq T}I_{1}(t)\leq C_{T},
\end{equation*}
for some constant $C_{T}>0$\,. For $I_{2}$, by the factorization method again, we have
\begin{equation*}
I_{3}=\tfrac{\sin\pi\theta}{\theta}\int_{0}^{t}(t-s)^{\theta-1}e^{A}(t-s)A^{\alpha}Y^{\e}(s)\,ds,
\end{equation*}
where $Y^{\e}$ is defined as
\begin{equation*}
Y^{\e}(s)=\tfrac{1}{\sqrt{\e}}\int_{0}^{s}(s-r)^{\theta}e^{A(s-r)}[f(\tfrac{r}{\e}, u(r))-\bar{f}(u(r))]\,dr.
\end{equation*}
Then
\begin{equation*}
\mathbb{E}\sup_{0\leq t\leq T}I_{2}(t)\leq C_{T}\int_{0}^{T}\mathbb{E}\|A^{\alpha}Y^{\e}(s)\|_{0}\,ds,
\end{equation*}
for some $C_{T}>0$\,. Notice that
\begin{eqnarray}
\|A^{\alpha}Y^{\e}(s)\|_{0}^{2}&=&\tfrac{1}{\e}\int_{0}^{l}\int_{0}^{s}\int_{0}^{s}(s-r)^{\theta}(s-\tau)^{\theta}A^{\alpha}e^{A(s-r)}[f(\tfrac{r}{\e}, u(r,x))-\bar{f}(u(r,x))]
 \times \nonumber\\&&{} A^{\alpha}e^{A(s-\tau)}[f(\tfrac{\tau}{\e}, u(\tau,x))-\bar{f}(u(\tau,x))]\,dr d\tau dx\,. \label{e:AY}
\end{eqnarray}
A standard discussion for the averaged equation  yields that 
\begin{equation*}
\sup_{0\leq t\leq T}\|u(t)\|_{1}^{2}\leq C_{T},
\end{equation*}
for some constant $C_{T}>0$\,. Then
$A^{\alpha}e^{A(s-r)}[f(\tfrac{r}{\e}, u(r,x))-\bar{f}(u(r,x))]\in\mathcal{F}_{0}^{r}$  and
$A^{\alpha}e^{A(s-\tau)}[f(\tfrac{\tau}{\e}, u(\tau,x))-\bar{f}(u(\tau,x))]\in\mathcal{F}_{\tau}^{\infty}$ and they are bounded  real-valued random variables for fixed $x\in (0, l)$. Applying a mixing property~\cite[Proposition 7.2.2]{EK86}  and choosing   positive parameters  $\alpha$ and $\theta $ so that $\alpha+\theta<1/2$\,, we have
\begin{equation*}
\mathbb{E}\|A^{\alpha}Y^{\e}(s)\|^{2}_{0}\leq C_{T}\,,\quad 0\leq s\leq T,
\end{equation*}
and then
\begin{equation*}
\mathbb{E}\sup_{0\leq t\leq T}I_{2}(t)\leq C_{T},
\end{equation*}
for some constant $C_{T}>0$\,. So for some $\alpha>0$,
\begin{equation*}
\mathbb{E}\|z^{\e}\|_{C(0, T; H^{\alpha/2} )}\leq C_{T}\,.
\end{equation*}
Furthermore, for $s$\,, $t$ with $0\leq s<t\leq T$\,,
\begin{eqnarray*}
&&\|z^{\e}(t)-z^{\e}(s)\|^{2}_{0}
\\&=&\frac{2}{\e}\left\|\int_{s}^{t} e^{A(t-r)}[f(r/\e, u^{\e}(r))-\bar{f}(u(r))]\,dr\right\|^{2}_{0}\\&&{}+\frac{2}{\e}\left\|(I-e^{A(t-s)})\int_{0}^{s}e^{A(s-r)}[f(r/\e, u^{\e}(r))-\bar{f}(u(r))]dr\right\|_{0}^{2}.
\end{eqnarray*}
Then via a similar discussion as that for (\ref{e:AY}),  we conclude that for some $0<\gamma<1$,
\begin{equation*}
\mathbb{E}\|z^{\e}(t)-z^{\e}(s)\|_{0}^{2}\leq C_{T}|t-s|^{\gamma},
\end{equation*}
which yields the tightness of the distributions of $z^{\e}$ in $C(0, T; H)$\,.

We decompose $z^{\e}=z_{1}^{\e}+z_{2}^{\e}$ so that 
\begin{equation*}
\dot{z}_{1}^{\e}=Az_{1}^{\e}+\tfrac{1}{\sqrt{\e}}[f(t/\e,
u)-\bar{f}(u)]\,,\quad z_{1}^{\e}(0)=0,
\end{equation*}
 and
 \begin{equation*}
\dot{z}_{2}^{\e}=Az_{2}^{\e}+\tfrac{1}{\sqrt{\e}}[f(t/\e,
u^{\e})-f(t/\e, u)]\,,\quad z_{2}^{\e}(0)=0\,.
\end{equation*}

For   $\phi\in C_{b}^{\infty}(0, l)$, we also consider the limit $\Phi(\langle z_{1}^{\e}(t), \phi\rangle )$ for every bounded second order differentiable function $\Phi: \mathbb{R}\rightarrow \mathbb{R}$ in the  weak convergence method.
Notice  that
\begin{eqnarray*}
&&\Phi(\langle z_{1}^{\e}(t), \varphi\rangle)-\Phi(\langle 0, \varphi\rangle)
=\int_{0}^{t}\Phi'(\langle z_{1}^{\e}(s), \varphi\rangle)\langle z_{1}^{\e}(s), \varphi_{xx}\rangle\,ds\\&&\qquad+\tfrac{1}{\sqrt{\e}}\int_{0}^{t}\Phi'(\langle z_{1}^{\e}(s), \varphi\rangle)\langle f(\tfrac{s}{\e}, u(s))-\bar{f}(u(s)), \varphi\rangle\,ds\,.
\end{eqnarray*}
Define the following process
\begin{eqnarray}
F^{\e}_{3}(t):=\tfrac{1}{\sqrt{\e}}\mathbb{E}\left[\int_{t}^{\infty}\Phi'(\langle z_{1}^{\e}(t), \varphi\rangle)  \langle f(\tfrac{s}{\e}, u(t))-\bar{f}(u(t)), \varphi\rangle\,ds \Big|\mathcal{F}^{t/\e} _{0}\right].
\end{eqnarray}
A direct calculation yields that 
\begin{eqnarray*}
&&F^{\e}(t):=A^{\e}\Phi(\langle z_{1}^{\e}(t), \phi\rangle)-A^{\e}F_{3}^{\e}(t)\\
&=& \Phi'(\langle z_{1}^{\e}(t), \phi \rangle)\langle z_{1}^{\e}(t), A\phi\rangle+\Phi''(\langle z_{1}^{\e}(t), \phi\rangle)\\&&{}\times \tfrac{1}{\e}\int_{t}^{\infty}\mathbb{E}[\langle f(\tfrac{t}{\e}, u(t))-\bar{f}(u(t)), \phi \rangle \langle f(\tfrac{s}{\e}, u(t))-\bar{f}(u(t)), \phi  \rangle|\mathcal{F}_{0}^{t/\e}] ds\\&&{}
+\Phi''(\langle z_{1}^{\e}(t), \phi\rangle)\langle z_{1}^{\e}(t), A\phi\rangle \tfrac{1}{\sqrt{\e}}\int_{t}^{\infty}\mathbb{E}[\langle f(\tfrac{s}{\e}, u(t))-\bar{f}(u(t)), \phi  \rangle|\mathcal{F}_{0}^{t/\e}] ds\,.
\end{eqnarray*}
Define two bilinear operators
\begin{equation*}
 B^{\e}(u, s, t) :=2[f(\tfrac{t}{\e}, u)-\bar{f}(u)]\otimes [f(\tfrac{s}{\e}, u)-\bar{f}(u)],
\end{equation*}
and
\begin{equation*}
\tilde{B}(u) :=2\int_0^\infty\mathbb{E}\left[ (f(t, u)-\bar{f}(u))\otimes
(f(0, u)-\bar{f}(u)) \right]\,dt\,.
\end{equation*}
Then by a mixing property~\cite[Proposition 7.2.2]{EK86}, we have
\begin{equation*}
\mathbb{E}|F_{3}^{\e}(t)|\rightarrow 0\,,
\end{equation*}
\begin{equation*}
\mathbb{E}\left|\Phi(\langle z^\e_{1}(t), \phi\rangle) \left[\tfrac{1}{\e}\int_{t}^{\infty}\mathbb{E}\left[\tfrac12\left\langle B^{\e}(u(t), s, t)\phi, \phi\right\rangle\big|\mathcal{F}_{0}^{t/\e} \right]ds -\tfrac12\left\langle \tilde{B}(u(t))\phi, \phi\right\rangle  \right]\right|\rightarrow 0,
\end{equation*}
and
\begin{equation*}
\mathbb{E}\left|\Phi''(\langle z_{1}^{\e}(t), \phi\rangle)\langle z_{1}^{\e}(t), A\phi\rangle \tfrac{1}{\sqrt{\e}}\int_{t}^{\infty}\mathbb{E}[\langle f(\tfrac{s}{\e}, u(t))-\bar{f}(u(t)), \phi  \rangle|\mathcal{F}_{0}^{t/\e}] ds\right|\rightarrow 0,
\end{equation*}
as $\e\rightarrow 0$\,.
Then we also have a  martingale
\begin{eqnarray*}
\mathcal{M}_{t}^{\e}&:=&\Phi(\langle z_{1}^{\e}(t), \phi\rangle) -\int_{0}^{t}\Phi'(\langle z_{1}^{\e}(s), \phi\rangle) \langle z_{1}^{\e}(s), \phi_{xx} \rangle \,ds\\&&{}
-\tfrac12\int_{0}^{t}\Phi''(\langle z_{1}^{\e}(s), \phi\rangle)\langle \tilde{B}(u) \phi, \phi\rangle\,ds+\mathcal{O}(\e)\,.
\end{eqnarray*}
By passing the limit $\e\rightarrow 0$\, and by the same discussion as in \S \ref{boundary888}, we see that  $z_{1}^{\e}$ converges in distribution to $z_{1}$, which solves
 \begin{equation}
 dz_{1}=Az_{1}+d\widetilde{W}\,,\quad z_{1}(0)=0,
 \end{equation}
 where $\widetilde{W}$ is an $H$-valued Wiener process defined on a new probability space $(\bar{\Omega}, \bar{\mathcal{F}}, \bar{\mathbb{P}})$ with covariance operator $\tilde{B}(u)$\,. Furthermore, $z_{2}^{\e}$ converges in distribution to  $z_{2}$, which solves 
 \begin{equation*}
 \dot{z}_{2}=Az_{2}+\overline{f'}(u)z\,,\quad z_{2}(0)=0\,.
 \end{equation*}
 Then $z^{\e}$ converges in distribution to $z$ with $z$ solving (\ref{e:zz})\,.
The proof is complete.
\end{proof}


\begin{remark}
The assumption on the strong mixing property in $(H)$ can be weakened as
\begin{equation*}
\int_0^\infty\phi^\alpha(t)\,dt<\infty,
\end{equation*}
for some $\alpha>0$\,.
In this case,  we also have Theorem
\ref{thm:AVE-DEV-RPDE}. See \cite{Watanabe, ParPia} for more details.
\end{remark}

\appendix

\section{Proof of  Lemma \ref{lem:A1}}\label{app:A1}

A similar result has been given in \cite[Proposition 7]{ParPia}. Here we present  another proof which gives a stronger  convergence, together with the convergence rate in probability.

First, under  the assumption ($\textbf{H}_{g}$),
we show that for almost all $\omega\in\Omega$\,,
\begin{equation}\label{e:Birkhoff}
\left\|\int_{0}^{t}\left[g(\tfrac{r}{\e}, q)-\bar{g}(q)\right]\,dr\right\|_{0}=\mathcal{O}(\sqrt{\e})\,,\quad \e\rightarrow 0,
\end{equation}
for every $q\in H$\,.

Noticing
\begin{equation*}
\int_{0}^{t}[g(r/\e, q)-\bar{g}(q)]\,dr\in H\,,
\end{equation*}
we get
\begin{equation*}
\left\|\int_{0}^{t}\left[g(\tfrac{r}{\e}, q)-\bar{g}(q)\right]\,dr\right\|_{0}=\sup_{\phi\in H}\frac{1}{\|\phi\|_{0}}\left|\left\langle  \int_{0}^{t}\left[g(\tfrac{r}{\e}, q)-\bar{g}(q)\right]\,dr, \phi\right\rangle\right|.
\end{equation*}
Consider
\begin{eqnarray*}
&&\mathbb{E}\left\langle  \int_{0}^{t}\left[g(\tfrac{r}{\e}, q)-\bar{g}(q)\right]\,dr, \phi\right\rangle^{2}\\&=& \int_{0}^{t}\int_{0}^{t}\mathbb{E}\left\langle g(\tfrac{r}{\e}, q)-\bar{g}(q), \phi\right\rangle\left\langle g(\tfrac{s}{\e}, q)-\bar{g}(q), \phi\right\rangle \,drds\,.
\end{eqnarray*}
By a mixing property~\cite[Proposition 7.2.2 ]{EK86}\,,  we have
\begin{equation*}
\mathbb{E}\left\langle  \int_{0}^{t}\left[g(\tfrac{r}{\e}, q)-\bar{g}(q)\right]\,dr, \phi\right\rangle^{2}=\mathcal{O}(\e)\|\phi\|_{0},
\end{equation*}
which yields (\ref{e:Birkhoff})\,.

 By the estimate in  \S \ref{sec:tight}\,, for  every $\kappa>0$\,, there is  $C_{T}^{\kappa}>0$, which is independent of $\e$, such that
\begin{equation}\label{e:holder}
\mathbb{P}\big\{\|\hat{u}^\e(t)-\hat{u}^{\e}(s) \|_{0}\leq C_{T}^{\kappa}\sqrt{t-s}\big\}\geq 1-\kappa,
\end{equation}
for every $t\geq s\geq 0$\,.
Furthermore, by the tightness of the distributions of $\{u^{\e}\}$ in space $C(0, T; H)$\,,  for every $\kappa>0$\,, there is a compact set $K_{\kappa}\subset C(0, T; H)$ such that
\begin{equation}\label{e:cmp}
\mathbb{P}\{u^{\e}\in K_{\kappa}\}\geq 1-\kappa\,.
\end{equation}
So we  define
\begin{equation*}
\Omega_{\kappa}=\{\omega\in\Omega: \text{events in}\, (\ref{e:Birkhoff})\,,  (\ref{e:holder})\, \text{and}\,  (\ref{e:cmp})\,  \text{hold}\}\,.
\end{equation*}
Due to the compactness of $K_{\kappa}$, for every $\varepsilon>0$, we   only need to consider a finite $\varepsilon$-net $\{q_{1}\,, q_{2}\,, \ldots\,, q_{N}\}$ in $C(0, T; H)$, which covers $\{u^{\e}\}$. Without loss of generality, we assume that $q_{j}$\,, $j=1,2,\ldots, N$\,, are simple functions~\cite{ParPia}.

 Now we consider all $\omega\in \Omega_{\kappa}$\,.   By the construction of $\hat{u}^{\e}$  and boundedness of $f$\,, we have for $\omega\in\Omega_{\kappa}$
 \begin{equation*}
 \|u^{\e}(t)-u^{\e}(s)\|_{0}\leq C_{T}^{\kappa}\sqrt{t-s}+\sqrt{\e} C,
 \end{equation*}
 for some constant $C>0$\,.

 For every $\delta>0$\,, we partition the interval $[0, T]$ into subintervals of length of $\delta$\,. Then for $t\in [k\delta, (k+1)\delta)$\,, $0\leq k\leq [\tfrac{T}{\delta}]$,
\begin{eqnarray*}
&&\left\|\int_{k\delta}^{t}\big[g(\tfrac{r}{\e}, u^{\e}(r))-\bar{g}(u^{\e}(r))\big]\,dr\right\|_{0}\\&\leq &
\left\|\int_{k\delta}^{t}\big[g(\tfrac{r}{\e}, u^{\e}(r))-g(\tfrac{r}{\e}, u^{\e}(k\delta))\big]\,dr\right\|_{0}+
\left\|\int_{k\delta}^{t}\big[g(\tfrac{r}{\e}, u^{\e}(k\delta))-g(\tfrac{r}{\e}, q_{j}(k\delta))\big]\,dr\right\|_{0}\\
&&{}+\left\|\int_{k\delta}^{t}\big[g(\tfrac{r}{\e}, q_{j}(k\delta))-\bar{g}(q_{j}(k\delta))\big]\,dr\right\|_{0}+\left\|\int_{k\delta}^{t}\big[\bar{g}(q_{j}(k\delta))-\bar{g}( u^{\e}(k\delta))\big]\,dr\right\|_{0}\\
&&{}+\left\|\int_{k\delta}^{t}\big[\bar{g}(u^{\e}(k\delta))-\bar{g}( u^{\e}(r))\big]\,dr\right\|_{0},
\end{eqnarray*}
for some $q_{j}$\,.
Notice that, by the assumption $(\textbf{H}_{g})$ and the definition of $\bar{g}$, $\bar{g}$ is also Lipschitz continuous in $u$ with the same Lipschitz constant $L_{g}$.  Then by the assumption $(\textbf{H}_{g})$ and  the definition of $\Omega_{\kappa}$\,,
\begin{eqnarray*}
&&\left\|\int_{0}^{t}\big[g(\tfrac{r}{\e}, u^{\e}(r))-\bar{g}(u^{\e}(r))\big]\,dr\right\|_{0}\\
&\leq&T[ L_{g}C_{T}^{\kappa}\delta+\sqrt{\e} C +L_{g}\varepsilon+ \mathcal{O}(\sqrt{\e})+L_{g}\varepsilon+L_{g}C_{T}^{\kappa}\delta+\sqrt{\e} C]\,.
\end{eqnarray*}
Due to the arbitrary choice of $\delta$\,, $\varepsilon$ and $\kappa$, and notice a similar discussion as that in \cite{ParPia},  we thus complete the proof.
\qed

\medskip

\textbf{Acknowledgement}.  This work was done while Jian Ren was visiting the Institute for Pure and Applied Mathematics (IPAM), Los Angeles, CA 90025, USA. It was partly supported by the NSF Grant  1025422, the Simons Foundation grant 208236,   the NSFC grants 10971225 and 11028102, an open grant of Laboratory for Nonlinear Mechanics at the   Chinese Academy of Sciences, and the Fundamental Research
Funds for the Central Universities (HUST  No.2010ZD037 and No.
2011QNQ170).


\end{document}